\documentclass{amsart}
\usepackage{amssymb}
\usepackage{amsfonts}
\usepackage{amsmath, amssymb}
\usepackage{hyperref}
\usepackage[displaymath, mathlines]{lineno}

\newtheorem{theorem}{Theorem}
\theoremstyle{plain}

\newtheorem{conjecture}{Conjecture}
\newtheorem{corollary}{Corollary}

\newtheorem{lemma}{Lemma}

\newtheorem{remark}{Remark}

\numberwithin{equation}{section}

\begin{document}
\title[A Class of Analytic
Functions]{A Class of Analytic
Functions associated with Sine Hyperbolic Functions}
\author[S.S. Kumar, M. G. Khan, B. Ahmad, M. K. Mashwani]{S. Sivaprasad Kumar$%
^{1,\ast },$ Muhammad Ghaffar Khan$^{2}$, Bakhtiar Ahmad$^{3},$ Wali Khan
Mashwani$^{2}$}
\address{$^{1,\ast }$Department of Applied Mathematics, Delhi College of
Engineering Bawana Road, Badli Delhi-110042, India}
\email{spkumar@dce.ac.in}
\address{$^{2}$Institute of Numerical Sciences, Kohat University of Science
and Technology, Kohat 26000, KPK, Pakistan}
\email{ghaffarkhan020@gmail.com}
\address{$^{3}$Department of Mathematics, Govt Degree College Mardan,
Pakistan}
\email{pirbakhtiarbacha@gmail.com}
\date{\\
\indent$^{* }$ Corresponding author\\
2010\textit{\ Mathematics Subject Classification. }30C45, 30D30.}

\begin{abstract}
We introduce a class of analytic functions subordinate to
the function $1+\sinh \left( z\right) $ and obtain various
necessary and sufficient conditions for functions to be in the class.
These conditions mainly comprise of the coefficient inequalities involving convolution. Further, we have obtained sharp five initial
coefficients,  a conjecture for the
general nth coefficient and the third Hankel determinant bounds for the functions in this class. Also derived certain differential subordination implication results involving $1+\sinh \left( z\right)$.
\end{abstract}

\maketitle

\section{Introduction and Definitions}

\noindent Let $\mathcal{A}$ be the class of all analytic functions $f\left(
z\right) $ defined in the open unit disc $\mathbb{D}=\left \{ z\in \mathbb{C}:\left
\vert z\right \vert <1\right \} $ with the power series representation as
\begin{equation}
f(z)=z+\sum \limits_{n=2}^{\infty }a_{n}z^{n}.
\label{eq1}
\end{equation}
Further, $\mathcal{S}$ denote the class of functions $f \in
\mathcal{A}$ that are univalent in $\mathbb{D}$. The function $f_{1}(z)$ is subordinate to $f_{2}(z)$, symbolically
written as $f_{1}(z)\prec f_{2}(z)$, if there exists a Schwarz function $\omega (z)$, $|\omega (z)|\leq |z|$, such that
$f_{1}(z)=f_{2}\left( \omega (z)\right) ,\; (z\in \mathbb{D}).$
Furthermore, if the function $f_{2}$ belongs to class $\mathcal{S}$, then
we have following equivalence condition
$f_{1}(z)\prec f_{2}(z),\; \left( z\in \mathbb{D}\right)$
if and only if $f\left( \mathbb{D}\right) \subseteq g\left( \mathbb{D}%
\right)$ and $f\left( 0\right) =g\left( 0\right)$.
For function $f$ of the form $\left( \ref{eq1}\right) $ and $g$ given by
\begin{equation*}
g\left( z\right) =z+\sum \limits_{n=2}^{\infty }b_{n}z^{n},
\end{equation*}
the Hadamard product or convolution of $f$ and $g$ is defined by
\begin{equation*}
\left( f\ast g\right) \left( z\right) =z+\sum \limits_{n=2}^{\infty
}a_{n}b_{n}z^{n}.
\end{equation*}
Recall that $f(z)=f(z)\ast\dfrac{z}{1-z}$ and  $zf'(z)=f(z)\ast\dfrac{z}{(1-z)^2}$.
Let $\mathcal{P}$ be the class of analytic functions $k(z)$ with
positive real part in $\mathbb{D}$ with the normalization
\begin{equation}
k(z)=1+\sum_{n=1}^{\infty }c_{n}z^{n}.  \label{p1}
\end{equation}
\noindent In 1992, Ma and Minda \cite{Ma} introduced and studied the following subclass of starlike functions in
$\mathcal{A}$:%
\begin{equation}
\mathcal{S}^{\ast }\left( h\right) =\left \{ f\in \mathcal{A}:\frac{%
zf^{\prime }(z)}{f(z)}\prec h(z)\prec\frac{1+z}{1-z},\text{ }z\in \mathbb{D}%
\right \} ,  \label{m3}
\end{equation}%
where $h$ has positive real part, $h(\mathbb{D})$ symmetric about the real axis with $h'(0)>0$ and $h(0)=1$.
Now by changing the function on the right hand side of $\left( \ref{m3}%
\right) $, we obtain several subclasses of the class $\mathcal{S}$, which
were introduced and investigated earlier, for example if we set $h\left(
z\right) =(1+Az)/(1+Bz),$ where $-1\leq B<A\leq 1,$ we obtain Janowski
class $\mathcal{S}^{\ast }\left[ A,B\right] ,$ see \cite{jan}. If $h\left(
z\right) =1+\sin \left( z\right) $, we obtain the class $\mathcal{S}_{\sin
}^{\ast }$ introduced by Cho et al. \cite{ch} and also see \cite{arifs}. By
setting \ $h(z)=\sqrt{1+z}$ we get the class $\mathcal{S}_{L}^{\ast },$ which was
introduced and studied by Sok\'{o}\l{} and Stankiewicz \cite{8.1} and further
studied by authors in \cite{8.2}. By varying $h\left( z\right), $
following classes are obtained:

\begin{enumerate}
\item If \ $h(z)=\cosh (z),$ Alotaibi et al. \cite{ar} introduced and
discussed class $\mathcal{S}_{\cosh }^{\ast }=\mathcal{S}^{\ast }\left(
\cosh \left( z\right) \right) $.

\item If \ $h(z)=1+\frac{4}{3}z+\frac{2}{3}z^{2},$ the class $\mathcal{S}%
_{Car}^{\ast }=\mathcal{S}^{\ast }\left( 1+\frac{4}{3}z+\frac{2}{3}%
z^{2}\right) $ associated with cardioid introduced by Sharma et al. \cite{9}.

\item If \ $h(z)=e^{z},$ the class $\mathcal{S}_{e}^{\ast }=\mathcal{S}%
^{\ast }\left( e^{z}\right) $ was introduced and studied by Mendiratta et
al. \cite{11} and further investigated by Shi et al. \cite{12}.

\item If \ $h(z)=z+\sqrt{1+z^{2}},$ Raina and Sokol et al. \cite{1}
introduced and discussed the class $\mathcal{S}_{\ell }^{\ast }=\mathcal{S}%
^{\ast }\left( z+\sqrt{1+z^{2}}\right) $ .

\item If \ $h(z)=\frac{2}{1+e^{-z}},$ recently the class was introduced and
discussed by Goel and Kumar \cite{sig}.

\item If \ $h(z)=1+z-\frac{1}{3}z^{3},$ more recently Wani and Swaminathan
introduced the class \cite{nep}.
\end{enumerate}

\noindent Also several subclasses of starlike functions were recently
introduced in \cite{14, 15, 16, 17, 18} by choosing a particular function
$h(z)$ such as functions associated with Bell numbers, functions associated
with shell-like curve connected with Fibonacci numbers or functions
connected with the conic domains.

 Kumar and Gangania~\cite{ganga} consider the analytic univalent function $\psi$ in $\mathbb{D}$ such that $\psi(0)=0$, $\psi(\mathbb{D})$ is
 starlike with respect to $0$ and introduced the following class of analytic functions:
\begin{equation}\label{gen-ma-min}
\mathcal{F}(\psi):= \left\{f\in \mathcal{A}: \frac{zf'(z)}{f(z)}-1 \prec \psi(z),\; \psi(0)=0 \right\}.
\end{equation}
Note that when $1+\psi(z)\not \prec (1+z)/(1-z)$, then the functions in the class $\mathcal{F}(\psi)$ may not be univalent in $\mathbb{D}$
which also implies $\mathcal{F}(\psi)\not\subseteq \mathcal{S}^{*}$ in general. Thus in case, when the function $1+\psi:=h$ has positive real
part, $h(\mathbb{D})$ symmetric about the real axis with $h'(0)>0$, then $\mathcal{F}(\psi)$ reduces to the class $\mathcal{S}^{*}(h)$. With
the condition that maximum and minimum of the real part of $\psi(z)$ is given by $\psi(\pm r)$, where $r=|z|$,  they established growth
theorem and obtained the sharp upper bound for distortion theorem for the class $\mathcal{F}(\psi)$. Hence improved the results which was
known for $0\leq \alpha \leq 3-2\sqrt{2}$ and $0\leq \beta \leq1/2$ for the following classes:
\begin{equation*}\label{boothlem}
\mathcal{BS}(\alpha):= \biggl\{f\in \mathcal{A} : \frac{zf'(z)}{f(z)}-1 \prec \frac{z}{1-\alpha z^2},\; \alpha\in [0,1) \biggl\},
\end{equation*}
where  $z/(1-\alpha z^2)=:\psi(z)$ is an analytic univalent function (known as Booth Lemniscate function) and symmetric with respect to the
real and imaginary axes and
\begin{equation*}\label{cissoidclass}
\mathcal{S}_{cs}(\beta):= \biggl\{f\in \mathcal{A} : \left(\frac{zf'(z)}{f(z)}-1\right) \prec \frac{z}{(1-z)(1+\beta z)},\; \beta\in [0,1)
\biggl\},
\end{equation*}
where $\frac{z}{(1-z)(1+\beta z)} := \psi(z)$ is univalent, analytic, symmetric about the real-axis and
maps the unit disk $\mathbb{D}$ onto the domain bounded by {\it Cissoid of Diocles}:
\begin{equation*}
CS(\beta):=\left\{ w=u+iv\in \mathbb{C}: \left(u-\frac{1}{2(\beta-1)}\right)(u^2+v^2)+ \frac{2\beta}{(1+\beta)^2(\beta-1)}v^2=0 \right\}
\end{equation*}
studied in \cite{kargar-2019}, \cite{NNEbadian-2018} and \cite{Masih-2019}.

 Motivated from the above, we introduce the subclass $\mathcal{G}_{sh
}$ of $\mathcal{F}(\psi)$ connected with a sine hyperbolic
function as:
\begin{equation*}
\mathcal{G}_{sh}:= \left\{f\in \mathcal{A}: \frac{zf'(z)}{f(z)}-1 \prec \sinh (z) \right\}.
\end{equation*}
Let $\phi(z):=1+\sinh(z)$. Note that $\phi \left( z\right)$ is not a Carath\'eodery function as ${\Re}\left( \phi \left( z\right) \right) \ngtr 0$ $%
\forall $ $z\in \mathbb{D}.$  A
function $f\in \mathcal{G}_{sh}$ if and only if there exists
an analytic function $q$ satisfying $q\left( z\right) \prec \phi \left(
z\right) $ such that
\begin{equation*}
f\left( z\right) =z\exp \left( \int_{0}^{z}\frac{q\left( t\right) -1}{t}dt\right) .
\end{equation*}
Thus choosing $q(z)=\phi(z)$, we have the following function
\begin{equation}\label{extremalform}
    f_0(z)= z\exp \left( \int_{0}^{z}\frac{\sinh(t)}{t}dt \right).
\end{equation}
As a consequence of \cite[Theorem~2.1, Corollary~2.1,~2.2, pg~3]{ganga}, we have the sharp results for the class $\mathcal{G}_{sh}$:
\begin{theorem}
Let $f\in \mathcal{G}_{sh}$ and $f_0$ be given as in \eqref{extremalform}. Then for $|z|=r$, we have
\begin{enumerate}
    \item $($growth theorem$)$ $-f_0(-r) \leq |f(z)| \leq f_0(r)$.
    \item $($covering theorem$)$ either $f$ is a rotation of $f_0$ or
		$$\{w\in{\mathbb{C}} : |w|\leq-{f}_0(-1) \} \subset f({\mathbb{D}}),$$
	where $-{f}_0(-1)=\lim_{r\rightarrow 1}(-f_0(-r)).$
	\item $\Re \dfrac{f(z)}{z} \leq \dfrac{f_0(r)}{r}$ and $|f'(z)|\leq \dfrac{(1+\sinh{r})f_0(r)}{r}$.
\end{enumerate}
\end{theorem}

In this paper, we consider some important properties like convolution
problems, necessary and sufficient conditions, coefficient problems, convex
combination, upper bounds for coefficients, Fekete-szeg\"{o} problems and third
Hankel determinant for the class $\mathcal{G}_{sh}$.

Let $f\in \mathcal{A}$, then qth Hankel determinant of $f$ is defined for
$q\geq 1,$ and $n\geq 1$ by
\begin{equation}
H_{q,n}\left( f\right) =\left \vert
\begin{array}{llll}
a_{n} & a_{n+1} & \ldots & a_{n+q-1} \\
a_{n+1} & a_{n+2} & \ldots & a_{n+q} \\
\vdots & \vdots & \ldots & \vdots \\
a_{n+q-1} & a_{n+q} & \ldots & a_{n+2q-2}
\end{array}
\right \vert .
\end{equation}
Thus second and third Hankel determinants are respectively:
\begin{align*}
&H_{2,2}\left( f\right) =a_{2}a_{4}-a_{3}^{2}, \\
&H_{3,1}\left( f\right) =a_{3}\left( a_{2}a_{4}-a_{3}^{2}\right)
-a_{4}\left( a_{4}-a_{2}a_{3}\right) +a_{5}\left( a_{3}-a_{2}^{2}\right) .
\end{align*}

\section{Preliminary}


\noindent The following lemmas are important for proving our
results.

\begin{lemma}
\label{L1}\cite{Ma}. If $k\in \mathcal{P}$ and it is of the form $\left( %
\ref{p1}\right) $, then for $\lambda \in \mathbb{C}$%
\begin{equation}
\left \vert c_{n}\right \vert \leq 2\text{ for }n\geq 1,  \label{1}
\end{equation}%
and%
\begin{equation}
\left \vert c_{2}-\lambda c_{1}^{2}\right \vert \leq 2\text{ }\max \left \{
1;\left \vert 2\lambda -1\right \vert \right \} \text{.}  \label{lc}
\end{equation}
\end{lemma}

\begin{lemma}
\label{L2}\cite{Ma}. If $k\in \mathcal{P}$ and is represented by $\left( %
\ref{p1}\right) $, then
\begin{equation*}
\left \vert c_{2}-\nu c_{1}^{2}\right \vert \leq \left \{
\begin{array}{ll}
-4\nu +2\quad & (\nu \leq 0), \\
2\quad & (0\leq \nu \leq 1), \\
4\nu -2\quad & (\nu \geq 1).%
\end{array}%
\right.
\end{equation*}
\end{lemma}

\begin{lemma}
\label{Lc}\cite{8p2, 8p3} If $k\in \mathcal{P}$ be expressed in series
expansion $\left( \ref{p1}\right) $, then%
\begin{equation*}
2c_{2}=c_{1}^{2}+x\left( 4-c_{1}^{2}\right)
\end{equation*}%
for some $x$, $\left \vert x\right \vert \leq 1$ and%
\begin{equation*}
4c_{3}=c_{1}^{3}+2\left( 4-c_{1}^{2}\right) c_{1}x-\left( 4-c_{1}^{2}\right)
c_{1}x^{2}+2\left( 4-c_{1}^{2}\right) \left( 1-\left \vert x\right \vert
^{2}\right) z
\end{equation*}%
for some $z$, $\left \vert z\right \vert \leq 1.$
\end{lemma}

\begin{lemma}
\label{Lj}If $k\in \mathcal{P}$ be expressed in series expansion $\left( %
\ref{p1}\right) $, then%
\begin{equation}
\left \vert ac_{1}^{3}-bc_{1}c_{2}+dc_{3}\right \vert \leq 2\left \vert
a\right \vert +2\left \vert b-2a\right \vert +2\left \vert a-b+d\right \vert
\label{ee1}
\end{equation}
\end{lemma}

\begin{lemma}
\label{La}\cite{Ravichandran} Let $m,n,l$ and $r$ satisfy the inequalities $%
0<m<1,0<r<1$ and
\begin{multline*}
  8r\left( 1-r\right) \left( \left( mn-2l\right) ^{2}+\left( m\left(
r+m\right) -n\right) ^{2}\right) +m\left( 1-m\right) \left( n-2rm\right) ^{2}\\
 \leq 4m^{2}\left( 1-m\right) ^{2}r\left( 1-r\right).
\end{multline*}

If $k\in \mathcal{P}$ and has power series $\left( \ref{p1}\right) $ then
\begin{equation*}
\left \vert lc_{1}^{4}+rc_{2}^{2}+2mc_{1}c_{3}-\frac{3}{2}%
nc_{1}^{2}c_{2}-c_{4}\right \vert \leq 2.
\end{equation*}
\end{lemma}

\begin{lemma}
\cite{jack} \label{R} Let $w\left( z\right) $ be analytic in $\mathbb{D}$ with $w\left( 0\right) =0.$ If $%
\left \vert w\left( z\right) \right \vert $ attains its maximum value on the
circle $\left \vert z\right \vert =r$ at a point $z_{0}=re^{i\theta },$ for $%
\theta \in \left[ -\pi ,\pi \right] ,$ we can write that%
\begin{equation*}
z_{0}w^{\prime }\left( z_{0}\right) =mw\left( z_{0}\right) ,
\end{equation*}%
where $m$ is real and $m\geq 1.$
\end{lemma}

\section{\textsc{Main Results}}
We begin with the following result:
\begin{theorem}
\label{Thc}Let $f\in \mathcal{A}$ be of the form $\left( \ref{eq1}\right) $%
. Then $f\in \mathcal{G}_{sh},$ if and only if%
\begin{equation}
\frac{1}{z}\left( f(z)\ast \frac{z-\beta z^{2}}{\left( 1-z\right) ^{2}}%
\right) \neq 0,  \label{bk1}
\end{equation}%
where $\beta =\beta _{\theta }=\dfrac{1+\sinh \left( e^{i\theta }\right) }{%
\sinh \left( e^{i\theta }\right) }.$
\end{theorem}

\begin{proof}
Let $f\in \mathcal{G}_{sh},$ if and only if
\begin{equation*}
\frac{zf^{\prime }(z)}{f(z)}\prec1+\sinh \left( z\right).
\end{equation*}%
if and only if there exist a Schwartz function $s(z)$ such that
\begin{align*}
& \frac{zf^{\prime }(z)}{f(z)}= 1+\sinh \left( s(z) \right)\quad
(z\in\mathbb{D})   \\
\Leftrightarrow&\frac{zf^{\prime }(z)}{f(z)}\neq 1+\sinh \left( e^{i\theta }\right) ,\quad
(z\in\mathbb{D};\theta \in \left[ 0,2\pi \right))   \\
\Leftrightarrow& \frac{1}{z}\big(zf^{\prime }(z)- f(z)\left( 1+\sinh \left( e^{i\theta }\right) \right)\big)\neq 0\\
\Leftrightarrow& \frac{1}{z}\left( f(z)\ast \frac{z-\beta z^{2}}{\left( 1-z\right) ^{2}}%
\right) \neq 0,
\end{align*}
where $\beta $ is as given above and that completes the proof.
\end{proof}
Note that the forward part of Theorem \ref{Thc} also holds for $\beta=1.$ As if $f\in \mathcal{G}_{sh},$ then $f$ is analytic in $\mathbb{D}$
and thus $f(z)/z\neq0.$
\begin{theorem}
\label{th2}Let $f\in \mathcal{A}$ be of the form $\left( \ref{eq1}\right) $.
Then necessary and sufficient condition for function $f\left( z\right) $
belong to class $\mathcal{G}_{sh}$ is that%
\begin{equation}
1-\sum_{n=2}^{\infty }\frac{n-\left( 1+\sinh \left( e^{i\theta }\right)
\right) }{\sinh \left( e^{i\theta }\right) }a_{n}z^{n-1}\neq 0.  \label{gk}
\end{equation}
\end{theorem}

\begin{proof}
In the light of above Theorem \ref{Thc}, we show that $\mathcal{G}_{sh
}$ if and only if
\begin{eqnarray*}
0 &\neq &\frac{1}{z}\left[ f(z)\ast \frac{z-\beta z^{2}}{\left( 1-z\right)
^{2}}\right] \\
&=&\frac{1}{z}\left( zf^{\prime }(z)-\beta \left( zf^{\prime
}(z)-f(z)\right) \right) \\
&=&1-\sum_{n=2}^{\infty }\left( \left( \beta -1\right) n-\beta \right)
a_{n}z^{n-1} \\
&=&1-\sum_{n=2}^{\infty }\frac{n-\left( 1+\sinh \left( e^{i\theta }\right)
\right) }{\sinh \left( e^{i\theta }\right) }a_{n}z^{n-1}.
\end{eqnarray*}%
Hence the proof completes.
\end{proof}

\begin{theorem}
\label{th1.1}Let $f\in \mathcal{A}$ and satisfies
\begin{equation}
\sum_{n=2}^{\infty }\left \vert \frac{n-\left( 1+\sinh \left( e^{i\theta
}\right) \right) }{\sinh \left( e^{i\theta }\right) }\right \vert \left
\vert a_{n}\right \vert <1,  \label{26}
\end{equation}%
then $f\in \mathcal{G}_{sh}.$
\end{theorem}

\begin{proof}
To show $f\in \mathcal{G}_{sh},$ we need to show $\left( \ref{gk}%
\right)$. Consider
\begin{eqnarray*}
\left \vert 1-\sum_{n=2}^{\infty }\left( \left( \beta -1\right) n-\beta
\right) a_{n}z^{n-1}\right \vert &>&1-\sum_{n=2}^{\infty }\left \vert \left(
\left( \beta -1\right) n-\beta \right) a_{n}z^{n-1}\right \vert \\
&=&1-\sum_{n=2}^{\infty }\left \vert \left( \left( \beta -1\right) n-\beta
\right) \right \vert \left \vert a_{n}\right \vert \left \vert z\right \vert
^{n-1} \\
&>&1-\sum_{n=2}^{\infty }\left \vert \left( \left( \beta -1\right) n-\beta
\right) \right \vert \left \vert a_{n}\right \vert \\
&=&1-\sum_{n=2}^{\infty }\left \vert \frac{n-\left( 1+\sinh \left(
e^{i\theta }\right) \right) }{\sinh \left( e^{i\theta }\right) }\right \vert
\left \vert a_{n}\right \vert >0,
\end{eqnarray*}%
so by Theorem \ref{th2}, $f\in \mathcal{G}_{sh}.$
\end{proof}

\begin{theorem}
The class $\mathcal{G}_{sh}$ is convex.
\end{theorem}


\begin{proof}
Let
\begin{equation*}
f_{i}(z)=z+\sum_{n=2}^{\infty }a_{n,i}z^{n},\text{ for }i=\left \{ 1,2\right
\} .
\end{equation*}%
We have to show that $\mu f_{1}(z)+\left( 1-\mu \right) f_{2}(z)\in \mathcal{G}_{sh}.$ As
\begin{eqnarray*}
&&\mu f_{1}(z)+\left( 1-\mu \right) f_{2}(z) \\
&=&z+\sum_{n=2}^{\infty }\left( \mu a_{n,1}+\left( 1-\mu \right)
a_{n,2}\right) z^{n}
\end{eqnarray*}%
Consider%
\begin{eqnarray*}
&&\sum_{n=2}^{\infty }\left \vert \frac{n-\left( 1+\sinh \left( e^{i\theta
}\right) \right) }{\sinh \left( e^{i\theta }\right) }\right \vert \left
\vert \mu a_{n,1}+\left( 1-\mu \right) a_{n,2}\right \vert \\
&\leq &\mu \sum_{n=2}^{\infty }\left \vert \frac{n-\left( 1+\sinh \left(
e^{i\theta }\right) \right) }{\sinh \left( e^{i\theta }\right) }\right \vert
\left \vert a_{n,1}\right \vert+\left( 1-\mu \right) \sum_{n=2}^{\infty }\left \vert \frac{n-\left(
1+\sinh \left( e^{i\theta }\right) \right) }{\sinh \left( e^{i\theta
}\right) }\right \vert \left \vert a_{n,2}\right \vert \\
&<&\mu +\left( 1-\mu \right) =1.
\end{eqnarray*}%
Thus by virtue of Theorem \ref{th1.1}, $\mu f_{1}(z)+\left( 1-\mu \right)
f_{2}(z)\in \mathcal{G}_{sh}.$
\end{proof}


\begin{theorem}
\label{Th1.1}Let $f\in \mathcal{G}_{sh}$ be of the form $\left( \ref%
{eq1}\right) $. Then
\begin{eqnarray*}
\left \vert a_{2}\right \vert &\leq &1, \\
\left \vert a_{3}\right \vert &\leq &\frac{1}{2}, \\
\left \vert a_{4}\right \vert &\leq &\frac{1}{3}, \\
\left \vert a_{5}\right \vert &\leq &\frac{1}{4}.
\end{eqnarray*}%
These inequalities are sharp respectively for
\begin{equation*}
f\left( z\right) =z\exp \int_{0}^{z}\frac{\sinh \left( t^{n-1}\right) }{t}dt%
\text{ for }n=2,3,4,5.
\end{equation*}
\end{theorem}

\begin{proof}
Since $f\in \mathcal{G}_{sh},$ then there exists an analytic
function $s(z),$ $\left \vert s(z)\right \vert <1$ and $s\left( 0\right) =0,$
such that
\begin{equation}
\frac{zf^{\prime }(z)}{f(z)}=1+\sinh \left( s(z)\right) .\quad  \label{th1}
\end{equation}%
Denote
\begin{equation*}
\Psi \left( s(z)\right) =1+\sinh \left( s(z)\right)
\end{equation*}%
and
\begin{equation}
k(z)=1+c_{1}z+c_{2}z^{2}+\cdots =\frac{1+s(z)}{1-s(z)}.  \label{eq1.11}
\end{equation}%
Obviously, the function $k\in \mathcal{P}$ and $s(z)=\frac{k(z)-1}{k(z)+1%
}$. This gives%
\begin{align}
1+\sinh \left( \frac{k(z)-1}{k(z)+1}\right) =& 1+\frac{1}{2}
c_{1}z+\left( \frac{1}{2}c_{2}-\frac{1}{4}c_{1}^{2}\right) z^{2}+\left(
\frac{7}{48}c_{1}^{3}-\frac{1}{2}c_{2}c_{1}+\frac{1}{2}c_{3}\right)
z^{3} \nonumber\\
&+\left( -\frac{3}{32}c_{1}^{4}+\frac{7}{16}c_{1}^{2}c_{2}-\frac{1}{2}%
c_{3}c_{1}-\frac{1}{4}c_{2}^{2}+\frac{1}{2}c_{4}\right) \allowbreak
z^{4}++\cdots .  \label{sin}
\end{align}%
And other side,%
\begin{eqnarray}
\frac{zf^{\prime }(z)}{f(z)} &=&1+a_{2}z+\left( 2a_{3}-a_{2}^{2}\right)
z^{2}+\left( 3a_{4}-3a_{2}a_{3}+a_{2}^{3}\right) z^{3}  \notag \\
&&+\left( 4a_{5}-2a_{3}^{2}-4a_{2}a_{4}+4a_{2}^{2}a_{3}-a_{2}^{4}\right)
z^{4}+\cdots .  \label{star1.1}
\end{eqnarray}%
On equating coefficients of $\left( \ref{sin}\right)$ and $\left( \ref%
{star1.1}\right) $, we get%
\begin{eqnarray}
a_{2} &=&\frac{1}{2}c_{1},  \label{coe1} \\
a_{3} &=&\frac{1}{4}c_{2},  \label{coe2} \\
a_{4} &=&\frac{1}{144}c_{1}^{3}-\frac{1}{24}c_{2}c_{1}+\frac{1}{6}c_{3},
\label{coe3} \\
a_{5} &=&-\frac{5}{1152}c_{1}^{4}+\frac{5}{192}c_{1}^{2}c_{2}-\frac{1}{24}%
c_{1}c_{3}-\frac{1}{32}c_{2}^{2}+\frac{1}{8}c_{4}.  \label{coe4}
\end{eqnarray}%
Using $\left( \ref{1}\right)$ with equations $\left( \ref{coe1}\right) $ and
$\left( \ref{coe2}\right) ,$ we get$\allowbreak $
\begin{equation*}
\left \vert a_{2}\right \vert \leq 1\text{ and }\left \vert a_{3}\right
\vert \leq \frac{1}{2}.
\end{equation*}%
Application of Lemma \ref{Lj} to equation $\left( \ref{coe3}\right) ,$ we
get
\begin{equation*}
\left \vert a_{4}\right \vert \leq \frac{1}{3}.
\end{equation*}%
Application of Lemma \ref{La} to equation $\left( \ref{coe4}\right) ,$ we get%
\begin{equation*}
\left \vert a_{5}\right \vert \leq \frac{1}{4}.
\end{equation*}
\end{proof}

\begin{conjecture}
Let $f\in \mathcal{G}_{sh}$ be of the form $\left( \ref{eq1}%
\right) $. Then$\allowbreak $ $\allowbreak $%
\begin{equation*}
\left \vert a_{n}\right \vert \leq \frac{1}{n-1}\text{ for }n\geq 2.
\end{equation*}
\end{conjecture}
As a consequence of above theorem, we have following remark:
\begin{remark}
Let $f\in \mathcal{G}_{sh}$ be of the form $\left( %
\ref{eq1}\right) $. Then for $\lambda\in\mathbb{C}$ $\allowbreak $ $\allowbreak $%
\begin{eqnarray*}
\left \vert a_{3}-\lambda a_{2}^{2}\right \vert &\leq &\frac{1}{2}\max \left
\{ 1,\left \vert 2\lambda -1\right \vert \right \}   \label{c1}
\end{eqnarray*}%
and%
\begin{equation*}
\left \vert a_{4}-a_{2}a_{3}\right \vert \leq \frac{1}{3}.  \label{c3}
\end{equation*}
\end{remark}
For $\lambda =1,$ the above remark gave:

\begin{corollary}
\label{c1.1}Let $f\in \mathcal{G}_{sh}$ be of the form $%
\left( \ref{eq1}\right) $. Then%
\begin{equation*}
\left \vert a_{3}-a_{2}^{2}\right \vert \leq \frac{1}{2}.
\end{equation*}%
The result is sharp.
\end{corollary}

\begin{theorem}
\label{Th1.3}Let $f\in \mathcal{G}_{sh}$ be of the form $%
\left( \ref{eq1}\right) $. Then
\begin{equation*}
\left \vert a_{2}a_{4}-a_{3}^{2}\right \vert \leq \frac{1}{36}.
\end{equation*}%
The result is sharp.
\end{theorem}

\begin{proof}
Since from $\left( \ref{coe1}\right) $, $\left( \ref{coe2}\right) $ and $%
\left( \ref{coe3}\right) $, we have
\begin{equation*}
\left \vert a_{2}a_{4}-a_{3}^{2}\right \vert =\left \vert \frac{1}{288}%
c_{1}^{4}-\frac{1}{48}c_{2}c_{1}^{2}+\frac{1}{12}c_{3}c_{1}-\frac{1}{16}%
c_{2}^{2}\right \vert .
\end{equation*}%
Using Lemma \ref{Lc}, also put $c_{1}=c\in \left[ 0,2\right], $  without loss of generality assume $|x|=y\in %
\left[ 0,1\right] $ and eliminating $z$ using triangular inequality, we get%
\begin{equation*}
\left \vert a_{2
}a_{4}-a_{3}^{2}\right \vert \leq  \psi \left( c,y\right),
\end{equation*}%
where $$\psi \left( c,y\right):=\frac{1}{288}\left(\frac{1%
}{2}c^{4}+6c^{2}y\left( 4-c^{2}\right) +6c^{2}y^{2}\left( 4-c^{2}\right)
+12c\left( 4-c^{2}\right) \left( 1-y^{2}\right) +\frac{9}{2}y^{2}\left(
4-c^{2}\right) ^{2}\right).$$
Now differentiating $\psi \left( c,y\right) $ with respect to $y$, we have
\begin{equation*}
\frac{\partial \psi \left( c,y\right) }{\partial y}=\frac{1}{288}\left(6c^{2}\left(
4-c^{2}\right) +12c^{2}y\left( 4-c^{2}\right) -24c\left( 4-c^{2}\right)
y+9y\left( 4-c^{2}\right) ^{2}\right),
\end{equation*}%
since $\dfrac{\partial \psi \left( c,y\right) }{\partial y}>0,$ thus $\psi(c,y)$ is an increasing
function and maximum occur at $y=1,$ so
\begin{equation*}
\psi \left( c,1\right) =\chi \left( c\right) =\frac{1}{288}\left( \frac{1}{2}%
c^{4}+12c^{2}\left( 4-c^{2}\right) +\frac{9}{2}\left( 4-c^{2}\right) ^{2}%
\right) ,
\end{equation*}%
differentiating $\chi \left( c\right) $ with respect to $c,$ we have
\begin{eqnarray*}
\chi ^{{\prime }}\left( c\right) &=&\frac{1}{288}\left( 2c^{3}+24c\left(
4-c^{2}\right) -24c^{3}-18c\left( 4-c^{2}\right) \right) \\
\chi ^{{{\prime \prime }}}\left( c\right) &=&\frac{1}{288}\left(
24 -84c^2 \right),
\end{eqnarray*}%
hence $\chi ^{{{\prime \prime }}}\left( c\right) <0$ for $c=2,$ maxima
exists at $c=2,$ we have
\begin{equation*}
\left \vert a_{2}a_{4}-a_{3}^{2}\right \vert \leq \frac{1}{36}.
\end{equation*}%
This bound is sharp for function defined as follow:
\begin{equation*}
f\left( z\right) =z\exp \left( \int_{0}^{z}\frac{\sinh \left( t\right) }{t}%
dt\right) =z+z^{2}+\frac{1}{2}z^{3}+\frac{2}{9}z^{4}+\cdots ,
\end{equation*}
which concludes the proof.
\end{proof}

\begin{theorem}
Let $f\in \mathcal{G}_{sh}$ be of the form $\left( \ref{eq1}%
\right) $. Then
\begin{equation*}
\left \vert H_{3,1}\left( f\right) \right \vert \leq \frac{1}{4}\simeq
\allowbreak 0.25.
\end{equation*}
\end{theorem}

\begin{proof}
Since we know that%
\begin{equation*}
H_{3,1}\left( f\right) =a_{3}\left( a_{4}a_{2}-a_{3}^{2}\right) -a_{4}\left(
a_{2}a_{3}-a_{4}\right) +a_{5}\left( a_{3}-a_{2}^{2}\right) ,
\end{equation*}%
using Theorem \ref{Th1.1} and \ref{Th1.3} and corollary \ref{c1.1}, along
with triangular inequality, we get the desired result.
\end{proof}

\section{\textsc{Differential Subordination}}

\begin{theorem}
\label{R1}For $-1\leq B<A\leq 1$ and $f \in \mathcal{A}$,
if the conditions
\begin{equation}
\left \vert \alpha \right \vert \geq \frac{A-B}{1+\cos 1-\sin 1-\left \vert
B\right \vert \left( 1+\sinh 1+\cosh 1\right) }  \label{a1}
\end{equation}%
\  and
\begin{equation*}
1+\alpha zf^{\prime }\left( z\right) \prec \frac{1+Az}{1+Bz},
\end{equation*}%
holds. Then%
\begin{equation*}
\frac{f\left( z\right) }{z}\prec 1+\sinh z.
\end{equation*}
\end{theorem}

\begin{proof}
Let us define a function%
\begin{equation}
p\left( z\right) =1+\alpha zf^{\prime }\left( z\right) ,  \label{j2}
\end{equation}%
also we consider%
\begin{equation}
\frac{f\left( z\right) }{z}=1+\sinh w(z).  \label{j1}
\end{equation}%
Now to prove our result, we only need to prove that $w(z)$ is a Schwarz function, that is $\left \vert
w\left( z\right) \right \vert <1$ for $|z|<1.$
Upon logarithmic differentiation of $\left( \ref{j1}\right) $ and using $%
\left( \ref{j2}\right) $ we obtain%
\begin{equation*}
p\left( z\right) =1+\alpha z\left( 1+\sinh w(z)+zw'(z)\cosh w(z)\right)
\end{equation*}%
and so%
\begin{eqnarray*}
\left \vert \frac{p\left( z\right) -1}{A-Bp\left( z\right) }\right \vert
&=&\left \vert \frac{\alpha \left( 1+\sinh w(z)+zw^{\prime }(z)\cosh
w(z)\right) }{A-B\left( 1+\alpha \left( 1+\sinh w(z)+zw^{\prime }(z)\cosh
w(z)\right) \right) }\right \vert \\
&=&\left \vert \frac{\alpha \left( 1+\sinh w(z)+zw^{\prime }(z)\cosh
w(z)\right) }{\left( A-B\right) -\alpha B\left( 1+\sinh w(z)+zw^{\prime
}(z)\cosh w(z)\right) }\right \vert.
\end{eqnarray*}%
Now if $w\left( z\right) $ attains its maximum value at some $z=z_{0}$ and $%
\left \vert w\left( z_{0}\right) \right \vert =1.$ Then by Lemma~\ref{R}  for $m\geq 1,$ we have, $\begin{array}{c}
z_{0}w^{\prime }\left( z_{0}\right) =mw\left( z_{0}\right)%
\end{array}
$ and $w\left( z_{0}\right) =e^{i\theta },$ for $\theta \in \left[ -\pi ,\pi %
\right],$ we get%
\begin{eqnarray*}
\left \vert \frac{p\left( z_{0}\right) -1}{A-Bp\left( z_{0}\right) }\right
\vert &=&\left \vert \frac{\alpha \left( 1+\sinh e^{i\theta }+mw\left(
z_{0}\right) \cosh e^{i\theta }\right) }{\left( A-B\right) -\alpha B\left(
1+\sinh e^{i\theta }+mw\left( z_{0}\right) \cosh e^{i\theta }\right) }\right
\vert \\
&\geq &\frac{\left \vert \alpha \right \vert \left( 1+m\left \vert \cosh
e^{i\theta }\right \vert -\left \vert \sinh e^{i\theta }\right \vert \right)
}{\left( A-B\right) +\left \vert \alpha \right \vert |B|\left( 1+\left \vert
\sinh e^{i\theta }\right \vert +m\left \vert \cosh e^{i\theta }\right \vert
\right) }.
\end{eqnarray*}%
Since
\begin{equation*}
\left \vert \cosh e^{i\theta }\right \vert ^{2}=\cosh ^{2}\left( \cos \theta
\right) \cos ^{2}\left( \sin \theta \right) +\sinh ^{2}\left( \cos \theta
\right) \sin ^{2}\left( \sin \theta \right) =\Psi \left( \theta \right)
\end{equation*}%
\begin{equation*}
\left \vert \sinh e^{i\theta }\right \vert ^{2}=\sinh ^{2}\left( \cos \theta
\right) \cos ^{2}\left( \sin \theta \right) +\cosh ^{2}\left( \cos \theta
\right) \sin ^{2}\left( \sin \theta \right) =\Theta \left( \theta \right),
\end{equation*}%
one can see that, if we let $\Psi ^{\prime }\left(
\theta \right) =0$ and $\Theta ^{\prime }\left( \theta \right) =0$ has the
roots $\theta =0,\pm \pi ,\pm \frac{\pi }{2}$ in $\left[ -\pi ,\pi \right] ,$
also $\Psi \left( \theta \right) $ and $\Theta \left( \theta \right) $ are
even functions in this interval so%
\begin{equation*}
\max \left \{ \Psi \left( \theta \right) \right \} =\Psi \left( 0\right)
=\Psi \left( \pi \right) =\cosh ^{2}\left( 1\right) ,
\end{equation*}%
\begin{equation*}
\min \left \{ \Psi \left( \theta \right) \right \} =\Psi \left( \frac{\pi }{2%
}\right) =\cos ^{2}\left( 1\right) ,
\end{equation*}%
\begin{equation*}
\max \left \{ \Theta \left( \theta \right) \right \} =\Theta \left( 0\right)
=\Theta \left( \pi \right) =\sinh ^{2}\left( 1\right) ,
\end{equation*}%
\begin{equation*}
\min \left \{ \Theta \left( \theta \right) \right \} =\Theta \left( \frac{%
\pi }{2}\right) =\sin ^{2}\left( 1\right) .
\end{equation*}%
From these, we obtain%
\begin{equation}
\cos \left( 1\right) \leq \left \vert \cosh e^{i\theta }\right \vert \leq
\cosh \left( 1\right) ,  \label{new1}
\end{equation}%
\begin{equation}
\sin \left( 1\right) \leq \left \vert \sinh e^{i\theta }\right \vert \leq
\sinh \left( 1\right) .  \label{new2}
\end{equation}%
Now we use $\left( \ref{new1}\right) $ and $\left( \ref{new2}\right) $ to
obtain%
\begin{equation*}
\left \vert \frac{p\left( z_{0}\right) -1}{A-Bp\left( z_{0}\right) }\right
\vert \geq \frac{\left \vert \alpha \right \vert \left( 1+m\cos \left(
1\right) -\sin \left( 1\right) \right) }{\left( A-B\right) +\left \vert
\alpha \right \vert |B|\left( 1+\sinh \left( 1\right) +m\cosh \left( 1\right)
\right) }.
\end{equation*}%
Let%
\begin{equation*}
\phi \left( m\right) =\frac{\left \vert \alpha \right \vert \left( 1+m\cos
\left( 1\right) -\sin \left( 1\right) \right) }{\left( A-B\right) +\left
\vert \alpha \right \vert |B|\left( 1+\sinh \left( 1\right) +m\cosh \left(
1\right) \right) },
\end{equation*}
which implies
\begin{equation*}
 \phi ^{\prime }\left( m\right) =\frac{\left \vert \alpha \right
\vert \cos \left( 1\right) \left( A-B\right) +\left \vert \alpha \right
\vert ^{2}B\left( \cos \left( 1\right) \left( 1+\sinh \left( 1\right)
\right) -\cosh \left( 1\right) \left( 1-\sin \left( 1\right) \right) \right)
}{\left( \left( A-B\right) +\left \vert \alpha \right \vert |B|\left( 1+\sinh
\left( 1\right) +m\cosh \left( 1\right) \right) \right) ^{2}}>0,
\end{equation*}
which shows that $\phi \left( m\right) $ is an increasing function and hence
it will have its minimum value at $m=1$ and so
\begin{equation*}
\left \vert \frac{p\left( z_{0}\right) -1}{A-Bp\left( z_{0}\right) }\right
\vert \geq \frac{\left \vert \alpha \right \vert \left( 1+\cos \left(
1\right) -\sin \left( 1\right) \right) }{\left( A-B\right) +\left \vert
\alpha \right \vert |B|\left( 1+\sinh \left( 1\right) +\cosh \left( 1\right)
\right) }
\end{equation*}%
now by$\left( \ref{a1}\right),$\ we have%
\begin{equation*}
\left \vert \frac{p\left( z_{0}\right) -1}{A-Bp\left( z_{0}\right) }\right
\vert \geq 1,
\end{equation*}%
which contradicts the fact that $%
\begin{array}{c}
p\left( z\right) \prec (1+Az)/(1+Bz)%
\end{array}$
and hence we get the desired result.
\end{proof}

\begin{corollary}
For $-1\leq B<A\leq 1,$ and $g \in \mathcal{A}$, then if the
following conditions
\begin{equation*}
\left \vert \alpha \right \vert \geq \frac{A-B}{1+\cos 1-\sin 1-\left \vert
B\right \vert \left( 1+\sinh 1+\cosh 1\right) }
\end{equation*}%
\ and
\begin{equation}
1+\frac{\alpha z^{2}g'\left( z\right) }{g\left( z\right) }\left( 2+\frac{%
zg^{\prime \prime }\left( z\right) }{g^{\prime }\left( z\right) }-\frac{%
zg^{\prime }\left( z\right) }{g\left( z\right) }\right) \prec \frac{1+Az}{%
1+Bz},  \label{i1}
\end{equation}%
holds, then $g \in \mathcal{G}_{sh}.$
\end{corollary}

\begin{proof}
Take
\begin{equation*}
l\left( z\right) =\frac{z^{2}g^{\prime }\left( z\right) }{g\left( z\right) },
\end{equation*}%
then we have%
\begin{equation*}
zl^{\prime }\left( z\right) =\frac{z^{2}g'\left( z\right) }{g\left( z\right) }%
\left( 2+\frac{zg^{\prime \prime }\left( z\right) }{g^{\prime }\left(
z\right) }-\frac{zg^{\prime }\left( z\right) }{g\left( z\right) }\right)
\end{equation*}%
and so by $\left( \ref{i1}\right), $ we get
\begin{equation*}
1+\alpha zl^{\prime }\left( z\right) \prec \frac{1+Az}{1+Bz}
\end{equation*}%
and hence by Theorem \ref{R1}, we get
\begin{equation*}
\frac{l\left( z\right) }{z}=\frac{zg^{\prime }\left( z\right) }{g\left(
z\right) }\prec 1+\sinh z,
\end{equation*}%
thus $g \in \mathcal{G}_{sh
}.$
\end{proof}

\begin{theorem} \label{t11}
For $-1\leq B<A\leq 1$ and $f \in \mathcal{A}$, then if the
conditions
\begin{equation}
\left \vert \alpha \right \vert \geq \frac{\left( A-B\right) \left( 1+\sinh
1\right) }{1+\cos 1-\sin 1-B\left( 1+\cosh 1+\sinh 1\right) }  \label{a2}
\end{equation}%
\  and%
\begin{equation}
1+\alpha \frac{zf^{\prime }\left( z\right) }{f\left( z\right) }\prec \frac{%
1+Az}{1+Bz},  \label{j3}
\end{equation}%
holds, then%
\begin{equation*}
\frac{f\left( z\right) }{z}\prec 1+\sinh z.
\end{equation*}
\end{theorem}

\begin{proof}
Let%
\begin{equation*}
p\left( z\right) =1+\alpha \frac{zf^{\prime }\left( z\right) }{f\left(
z\right)}.
\end{equation*}%
and
\begin{equation*}
\frac{f\left( z\right) }{z}=1+\sinh w\left( z\right),
\end{equation*} then we have to show that $|w(z)|<1$ for $|z|<1$. Now
using simple calculations, we obtain that%
\begin{equation*}
p\left( z\right) =1+\alpha \frac{1+zw^{\prime }\left( z\right) \cosh w\left(
z\right) +\sinh w\left( z\right) }{1+\sinh w\left( z\right) }
\end{equation*}%
and so%
\begin{equation*}
\left \vert \frac{p\left( z\right) -1}{A-Bp\left( z\right) }\right \vert
=\left \vert \frac{\alpha \left( 1+zw^{\prime }\left( z\right) \cosh w\left(
z\right) +\sinh w\left( z\right) \right) }{\left( A-B\right) \left( 1+\sinh
w\left( z\right) \right) -B\alpha \left( 1+zw^{\prime }\left( z\right) \cosh
w\left( z\right) +\sinh w\left( z\right) \right) }\right \vert.
\end{equation*}%
On the contrary if $w\left( z\right) $ attains its maximum value at some $z=z_{0}$ and $%
\left \vert w\left( z_{0}\right) \right \vert =1.$ Then by Lemma \ref{R} for $m\geq 1$, we have$,$ $%
\begin{array}{c}
z_{0}w^{\prime }\left( z_{0}\right) =mw\left( z_{0}\right),%
\end{array}%
$ so we have%
\begin{eqnarray*}
\left \vert \frac{p\left( z_{0}\right) -1}{A-Bp\left( z_{0}\right) }\right
\vert &=&\left \vert \frac{\alpha \left( 1+mw\left( z_{0}\right) \cosh
e^{i\theta }+\sinh e^{i\theta }\right) }{\left( A-B\right) \left( 1+\sinh
e^{i\theta }\right) -B\alpha \left( 1+mw\left( z_{0}\right) \cosh e^{i\theta
}+\sinh e^{i\theta }\right) }\right \vert \\
&\geq &\frac{\left \vert \alpha \right \vert \left( 1+m\left \vert \cosh
e^{i\theta }\right \vert -\left \vert \sinh e^{i\theta }\right \vert \right)
}{\left( A-B\right) \left( 1+\left \vert \sinh e^{i\theta }\right \vert
\right) +|B|\left \vert \alpha \right \vert \left( 1+m\left \vert \cosh
e^{i\theta }\right \vert +\left \vert \sinh e^{i\theta }\right \vert \right)
}.
\end{eqnarray*}%
Now with the help of $\left( \ref{new1}\right) $ and $\left( \ref{new2}%
\right) $%
\begin{equation*}
\left \vert \frac{p\left( z_{0}\right) -1}{A-Bp\left( z_{0}\right) }\right
\vert \geq \frac{\left \vert \alpha \right \vert \left( 1+m\cos \left(
1\right) -\sin \left( 1\right) \right) }{\left( A-B\right) \left( 1+\sinh
\left( 1\right) \right) +|B|\left \vert \alpha \right \vert \left( 1+m\cosh
\left( 1\right) +\sinh \left( 1\right) \right) }.
\end{equation*}%
Now if
\begin{equation*}
\phi \left( m\right) =\frac{\left \vert \alpha \right \vert \left( 1+m\cos
\left( 1\right) -\sin \left( 1\right) \right) }{\left( A-B\right) \left(
1+\sinh \left( 1\right) \right) +|B|\left \vert \alpha \right \vert \left(
1+m\cosh \left( 1\right) +\sinh \left( 1\right) \right) }
\end{equation*}%
then
\begin{equation*}\phi ^{\prime }\left( m\right) = \frac{2\left( A-B\right) \left
\vert \alpha \right \vert \cos 1 \left( \sinh 1+1\right) +|B||\alpha|^2
\left( \cos 1\left( 1+\sinh 1\right) +\cosh 1\left(
\sin 1-1\right) \right) }{\left( \left( A-B\right) \left( 1+\sinh \left(
1\right) \right) +|B|\left \vert \alpha \right \vert \left( 1+m\cosh \left(
1\right) +\sinh \left( 1\right) \right) \right)^{2}} > 0,
\end{equation*}%
$\allowbreak \allowbreak $which shows that $\phi \left( m\right) $ is an
increasing function and hence it will have its minimum value at $m=1$ and so
\begin{equation*}
\left \vert \frac{p\left( z_{0}\right) -1}{A-Bp\left( z_{0}\right) }\right
\vert \geq \frac{\left \vert \alpha \right \vert \left( 1+\cos \left(
1\right) -\sin \left( 1\right) \right) }{\left( A-B\right) \left( 1+\sinh
\left( 1\right) \right) +B\left \vert \alpha \right \vert \left( 1+\cosh
\left( 1\right) +\sinh \left( 1\right) \right) }
\end{equation*}%
now by$\left( \ref{a2}\right), $\ we have%
\begin{equation*}
\left \vert \frac{p\left( z_{0}\right) -1}{A-Bp\left( z_{0}\right) }\right
\vert \geq 1,
\end{equation*}%
which contradicts $\left( \ref{j3}\right)$ and so $\left \vert w\left(
z\right) \right \vert <1$ for $|z|<1$, which completes the  proof.
\end{proof}
We obtain the following result using Theorem~\ref{t11}
\begin{corollary}
For $-1\leq B<A\leq 1,$ and $f \in \mathcal{A}_{p}$ then if
the condition
\begin{equation*}
\left \vert \alpha \right \vert \geq \frac{\left( A-B\right) \left( 1+\sinh
1\right) }{1+\cos 1-\sin 1-B\left( 1+\cosh 1+\sinh 1\right) },
\end{equation*}%
\ and
\begin{equation*}
1+\alpha \left( 2+\frac{zg^{\prime \prime }\left( z\right) }{g^{\prime
}\left( z\right) }-\frac{zg^{\prime }\left( z\right) }{g\left( z\right) }%
\right) \prec \frac{1+Az}{1+Bz},
\end{equation*}%
holds then $g \in \mathcal{G}_{sh}.$
\end{corollary}

\begin{theorem}
For $-1\leq B<A\leq 1$ and $f \in \mathcal{M}_{p}$ then if
the condition
\begin{equation}
\left \vert \alpha \right \vert \geq \frac{\left( A-B\right) \left( 1+\sinh
\left( 1\right) \right) ^{2}}{1+\cos \left( 1\right) -\sin \left( 1\right)
-B\left( 1+\cosh \left( 1\right) +\sinh \left( 1\right) \right) }  \label{a3}
\end{equation}%
\ is true and
\begin{equation}
1+\alpha \frac{z^{2}f^{\prime }\left( z\right) }{\left( f\left( z\right)
\right) ^{2}}\prec \frac{1+Az}{1+Bz},  \label{j4}
\end{equation}%
then%
\begin{equation*}
\frac{f\left( z\right) }{z}\prec \sqrt{1+z}.
\end{equation*}
\end{theorem}

\begin{proof}
Let%
\begin{equation*}
p\left( z\right) =1+\alpha \frac{z^{2}f^{\prime }\left( z\right) }{\left(
f\left( z\right) \right) ^{2}}
\end{equation*}%
and
\begin{equation*}
\frac{f\left( z\right) }{z}=1+\sinh w\left( z\right),
\end{equation*}
then it is to show that $|w|<1$ for $|z|<1$. Also
by simplification, we have%
\begin{equation*}
p\left( z\right) =1+\alpha \frac{1+zw^{\prime }\left( z\right) \cosh w\left(
z\right) +\sinh w\left( z\right) }{\left( 1+\sinh w\left( z\right) \right)
^{2}}
\end{equation*}%
and so%
\begin{equation*}
\left \vert \frac{p\left( z\right) -1}{A-Bp\left( z\right) }\right \vert
=\left \vert \frac{\alpha \left( 1+zw^{\prime }\left( z\right) \cosh w\left(
z\right) +\sinh w\left( z\right) \right) }{\left( A-B\right) \left( 1+\sinh
w\left( z\right) \right) ^{2}-B\alpha \left( 1+zw^{\prime }\left( z\right)
\cosh w\left( z\right) +\sinh w\left( z\right) \right) }\right \vert .
\end{equation*}%
Now if $w\left( z\right) $ attains its maximum value at some $z=z_{0}$ and $%
\left \vert w\left( z_{0}\right) \right \vert =1.$ Then by Lemma \ref{R} for $m\geq 1,$ we have, $%
\begin{array}{c}
z_{0}w^{\prime }\left( z_{0}\right) =mw\left( z_{0}\right).%
\end{array}%
$ So we have%
\begin{eqnarray*}
\left \vert \frac{p\left( z_{0}\right) -1}{A-Bp\left( z_{0}\right) }\right
\vert &=&\left \vert \frac{\alpha \left( 1+mw\left( z_{0}\right) \cosh
e^{i\theta }+\sinh e^{i\theta }\right) }{\left( A-B\right) \left( 1+\sinh
e^{i\theta }\right) ^{2}-B\alpha \left( 1+mw\left( z_{0}\right) \cosh
e^{i\theta }+\sinh e^{i\theta }\right) }\right \vert \\
&\geq &\frac{\left \vert \alpha \right \vert \left( 1+m\left \vert \cosh
e^{i\theta }\right \vert -\left \vert \sinh e^{i\theta }\right \vert \right)
}{\left( A-B\right) \left( 1+\left \vert \sinh e^{i\theta }\right \vert
\right) ^{2}+|B|\left \vert \alpha \right \vert \left( 1+m\left \vert \cosh
e^{i\theta }\right \vert +\left \vert \sinh e^{i\theta }\right \vert \right)
}.
\end{eqnarray*}%
Now using $\left( \ref{new1}\right) $ and $\left( \ref{new2}%
\right) $%
\begin{equation*}
\left \vert \frac{p\left( z_{0}\right) -1}{A-Bp\left( z_{0}\right) }\right
\vert \geq \frac{\left \vert \alpha \right \vert \left( 1+m\cos \left(
1\right) -\sin \left( 1\right) \right) }{\left( A-B\right) \left( 1+\sinh
\left( 1\right) \right) ^{2}+|B|\left \vert \alpha \right \vert \left(
1+m\cosh \left( 1\right) +\sinh \left( 1\right) \right) }.
\end{equation*}%
Let%
\begin{equation*}
\phi \left( m\right) =\frac{\left \vert \alpha \right \vert \left( 1+m\cos
\left( 1\right) -\sin \left( 1\right) \right) }{\left( A-B\right) \left(
1+\sinh \left( 1\right) \right) ^{2}+|B|\left \vert \alpha \right \vert \left(
1+m\cosh \left( 1\right) +\sinh \left( 1\right) \right) }
\end{equation*}%
which implies
\begin{equation*}
 \phi ^{\prime }\left( m\right) =\frac{\left( A-B\right) \left
\vert \alpha \right \vert \cos 1\left( 1+\sinh \left( 1\right) \right)
^{2}+|B|\left \vert \alpha \right \vert ^{2}\left( \cos 1\left( 1+\sinh
1\right) +\sin 1\cosh 1-\cosh 1\right) }{\left( \left( A-B\right) \left(
1+\sinh \left( 1\right) \right) ^{2}+|B|\left \vert \alpha \right \vert \left(
1+m\cosh \left( 1\right) +\sinh \left( 1\right) \right) \right) ^{2}}>0,
\end{equation*}%
$\allowbreak \allowbreak $which shows that $\phi \left( m\right) $ is an
increasing function and hence it will have its minimum value at $m=1$ and so
\begin{equation*}
\left \vert \frac{p\left( z_{0}\right) -1}{A-Bp\left( z_{0}\right) }\right
\vert \geq \frac{\left \vert \alpha \right \vert \left( 1+\cos \left(
1\right) -\sin \left( 1\right) \right) }{\left( A-B\right) \left( 1+\sinh
\left( 1\right) \right) ^{2}+|B|\left \vert \alpha \right \vert \left( 1+\cosh
\left( 1\right) +\sinh \left( 1\right) \right)}.
\end{equation*}%
Now by$\left( \ref{a3}\right) $\ we have%
\begin{equation*}
\left \vert \frac{p\left( z_{0}\right) -1}{A-Bp\left( z_{0}\right) }\right
\vert \geq 1,
\end{equation*}%
which contradicts $\left( \ref{j4}\right) ,$ thus $\left \vert w\left(
z\right) \right \vert <1$ for $|z|<1$, which yields the desired result.
\end{proof}

\begin{corollary}
For $-1\leq B<A\leq 1,$ and $g\in \mathcal{A}_{p}$ then if
the condition
\begin{equation*}
\left \vert \alpha \right \vert \geq \frac{\left( A-B\right) \left( 1+\sinh
\left( 1\right) \right) ^{2}}{1+\cos \left( 1\right) -\sin \left( 1\right)
-B\left( 1+\cosh \left( 1\right) +\sinh \left( 1\right) \right) },
\end{equation*}%
\ and
\begin{equation*}
1+\frac{\alpha g\left( z\right) }{zg^{\prime }\left( z\right) }\left( 2+%
\frac{zg^{\prime \prime }\left( z\right) }{g^{\prime }\left( z\right) }-%
\frac{zg^{\prime }\left( z\right) }{g\left( z\right) }\right) \prec \frac{%
1+Az}{1+Bz},
\end{equation*}%
holds then $g \in \mathcal{G}_{sh}.$
\end{corollary}

\begin{theorem}
For $-1\leq B<A\leq 1$ and $f \in \mathcal{M}_{p}$ then if
the condition
\begin{equation}
\left \vert \alpha \right \vert \geq \frac{\left( A-B\right) \left( 1+\sinh
\left( 1\right) \right) ^{3}}{1+\cos \left( 1\right) -\sin \left( 1\right)
-B\left( 1+\cosh \left( 1\right) +\sinh \left( 1\right) \right) },
\label{a4}
\end{equation}%
\ holds and
\begin{equation}
1+\alpha \frac{z^{3}f^{\prime }\left( z\right) }{\left( f\left( z\right)
\right) ^{3}}\prec \frac{1+Az}{1+Bz},\label{j5}
\end{equation}%
then%
\begin{equation*}
\frac{f\left( z\right) }{z}\prec 1+\sinh w\left( z\right) .
\end{equation*}
\end{theorem}

\begin{proof}
Let us assume%
\begin{equation*}
p\left( z\right) =1+\alpha \frac{z^{3}f^{\prime }\left( z\right) }{\left(
f\left( z\right) \right) ^{3}}
\end{equation*}%
and
\begin{equation*}
\frac{f\left( z\right) }{z}=1+\sinh w\left( z\right),
\end{equation*}
then, we need to show that $|w|<1$ for $|z|<1$.
Also by rearrangement, we get%
\begin{equation*}
p\left( z\right) =1+\alpha \frac{1+zw^{\prime }\left( z\right) \cosh w\left(
z\right) +\sinh w\left( z\right) }{\left( 1+\sinh w\left( z\right) \right)
^{3}}
\end{equation*}%
and so%
\begin{equation*}
\left \vert \frac{p\left( z\right) -1}{A-Bp\left( z\right) }\right \vert
=\left \vert \frac{\alpha \left( 1+zw^{\prime }\left( z\right) \cosh w\left(
z\right) +\sinh w\left( z\right) \right) }{\left( A-B\right) \left( 1+\sinh
w\left( z\right) \right) ^{3}-B\alpha \left( 1+zw^{\prime }\left( z\right)
\cosh w\left( z\right) +\sinh w\left( z\right) \right) }\right \vert .
\end{equation*}%
Now if $w\left( z\right) $ attains its maximum value at some $z=z_{0}$ and $%
\left \vert w\left( z_{0}\right) \right \vert =1.$ Then by Lemma  \ref%
{R} for $m\geq 1,$ we have$,$ $%
\begin{array}{c}
z_{0}w^{\prime }\left( z_{0}\right) =mw\left( z_{0}\right).%
\end{array}%
$So we have%
\begin{eqnarray*}
\left \vert \frac{p\left( z_{0}\right) -1}{A-Bp\left( z_{0}\right) }\right
\vert &=&\left \vert \frac{\alpha \left( 1+mw\left( z_{0}\right) \cosh
e^{i\theta }+\sinh e^{i\theta }\right) }{\left( A-B\right) \left( 1+\sinh
e^{i\theta }\right) ^{3}-B\alpha \left( 1+mw\left( z_{0}\right) \cosh
e^{i\theta }+\sinh e^{i\theta }\right) }\right \vert \\
&\geq &\frac{\left \vert \alpha \right \vert \left( 1+m\left \vert \cosh
e^{i\theta }\right \vert -\left \vert \sinh e^{i\theta }\right \vert \right)
}{\left( A-B\right) \left( 1+\left \vert \sinh e^{i\theta }\right \vert
\right) ^{3}+|B|\left \vert \alpha \right \vert \left( 1+m\left \vert \cosh
e^{i\theta }\right \vert +\left \vert \sinh e^{i\theta }\right \vert \right)
}.
\end{eqnarray*}%
Now with the help of $\left( \ref{new1}\right) $ and $\left( \ref{new2}%
\right) $%
\begin{equation*}
\left \vert \frac{p\left( z_{0}\right) -1}{A-Bp\left( z_{0}\right) }\right
\vert \geq \frac{\left \vert \alpha \right \vert \left( 1+m\cos \left(
1\right) -\sin \left( 1\right) \right) }{\left( A-B\right) \left( 1+\sinh
\left( 1\right) \right) ^{3}+|B|\left \vert \alpha \right \vert \left(
1+m\cosh \left( 1\right) +\sinh \left( 1\right) \right) }.
\end{equation*}%
Suppose%
\begin{equation*}
\phi \left( m\right) =\frac{\left \vert \alpha \right \vert \left( 1+m\cos
\left( 1\right) -\sin \left( 1\right) \right) }{\left( A-B\right) \left(
1+\sinh \left( 1\right) \right) ^{3}+|B|\left \vert \alpha \right \vert \left(
1+m\cosh \left( 1\right) +\sinh \left( 1\right) \right) },
\end{equation*}%
that implies
\begin{equation*}\small
\phi ^{\prime }\left( m\right) =\frac{\left( A-B\right) \left
\vert \alpha \right \vert \cos 1\left( 1+\sinh \left( 1\right) \right)
^{2}+|B|\left \vert \alpha \right \vert ^{2}\left( \cos 1+\cos 1\sinh 1+\sin
1\cosh 1-\cosh 1\right) }{\left( \left( A-B\right) \left( 1+\sinh \left(
1\right) \right) ^{3}+|B|\left \vert \alpha \right \vert \left( 1+m\cosh
\left( 1\right) +\sinh \left( 1\right) \right) \right) ^{2}}>0,
\end{equation*}%
$\allowbreak \allowbreak $which shows that $\phi \left( m\right) $ is an
increasing function and hence it will have its minimum value at $m=1$, thus
\begin{equation*}
\left \vert \frac{p\left( z_{0}\right) -1}{A-Bp\left( z_{0}\right) }\right
\vert \geq \frac{\left \vert \alpha \right \vert \left( 1+\cos \left(
1\right) -\sin \left( 1\right) \right) }{\left( A-B\right) \left( 1+\sinh
\left( 1\right) \right) ^{3}+|B|\left \vert \alpha \right \vert \left( 1+\cosh
\left( 1\right) +\sinh \left( 1\right) \right) }.
\end{equation*}%
Now by$\left( \ref{a4}\right) $\ we have%
\begin{equation*}
\left \vert \frac{p\left( z_{0}\right) -1}{A-Bp\left( z_{0}\right) }\right
\vert \geq 1,
\end{equation*}%
which contradicts $\left( \ref{j5}\right) ,$ therefore $\left \vert w\left(
z\right) \right \vert <1$ for $|z|<1$, which gives the desired result.
\end{proof}

\begin{corollary}
For $-1\leq B<A\leq 1,$ and $g \in \mathcal{A}_{p}$, then if
the condition%
\begin{equation*}
\left \vert \alpha \right \vert \geq \frac{\left( A-B\right) \left( 1+\sinh
\left( 1\right) \right) ^{3}}{1+\cos \left( 1\right) -\sin \left( 1\right)
-B\left( 1+\cosh \left( 1\right) +\sinh \left( 1\right) \right) }
\end{equation*}%
\  and
\begin{equation*}
1+\alpha \frac{g\left( z\right) ^{2}}{z^{2}\left( g^{\prime }\left( z\right)
\right) ^{2}}\left( 2+\frac{zg^{\prime \prime }\left( z\right) }{g^{\prime
}\left( z\right) }-\frac{zg^{\prime }\left( z\right) }{g\left( z\right) }%
\right) \prec \frac{1+Az}{1+Bz},
\end{equation*}%
holds, then $g\in \mathcal{G}_{sh}.$
\end{corollary}

\begin{center}
\textbf{Funding}
\end{center}

Not applicable.

\begin{center}
\textbf{Availability of data and materials}
\end{center}

Not applicable.

\begin{center}
\textbf{Competing interests}
\end{center}

\noindent The authors declare that they have no competing interests.

\begin{center}
\textbf{Authors contributions}
\end{center}

\noindent All authors jointly worked on the results and they read and
approved the final manuscript.

\end{document}